\documentclass[12pt,leqno]{amsart}
\usepackage[a4paper,top=2cm,bottom=2cm,left=2cm,right=2cm]{geometry}

\usepackage{natbib}
\usepackage{chngcntr}
\usepackage{lineno,hyperref}
\usepackage{rotating}
\usepackage{amsfonts}% per usufruire di mathbb{}%
\usepackage{amsmath}
\usepackage{amsthm}
\usepackage{xcolor}
\usepackage{amssymb}
\usepackage{subfig}
\usepackage[english]{babel}
\usepackage{graphicx}
\usepackage{tabularx}
\usepackage{enumerate, hyperref}

\usepackage{tikz}
\usetikzlibrary{shapes.geometric,calc}

\DeclareMathOperator{\sinc}{sinc}
\DeclareMathOperator{\rect}{rect}

\def\G{\mathcal G}
\def\B{\mathcal B}
\def\C{\mathbb C}

\def\F{\mathcal F}

\newcommand\Z{\mathbb Z}
\def\ZZ{\mathbb Z}

\def\R{\mathbb R}

\theoremstyle{plain}
\newtheorem{theorem}{Theorem}[section]
\newtheorem{corollary}[theorem]{Corollary}
\newtheorem{lemma}[theorem]{Lemma}

\theoremstyle{definition}

\theoremstyle{remark}

\title{Stability results for Gabor frames and the $p$-order hold models.}

\author[Laura De Carli, Pierluigi Vellucci]{Laura De Carli $^*$
   \and Pierluigi Vellucci$^{**}$}

\begin{document}
\maketitle

\begin{abstract}
We prove stability results for a class of  Gabor  frames in $ L^2(\R)$. We consider window functions  in the Sobolev spaces $H^1_0(\R)$ and   B-splines  of order $p\ge 1$. Our  results   can be used to describe the effect of the timing jitters in the $p$-order hold models  of  signal reconstruction.
\end{abstract}

\section{Introduction}
\label{sec:intro}

In{\let\thefootnote\relax\footnotetext{Keywords: Gabor frames, B-splines, signal reconstruction, timing jitters, stability results. 2010 MSC: 46A35, 41A04.}} this{\let\thefootnote\relax\footnotetext{$^*$Dept. of Mathematics, Florida International University, University Park, 11200 SW 8th str., Miami, FL 33199, USA (decarlil@fiu.edu). $^{**}$ Dept. of Economics, Roma Tre University, via Silvio D'Amico 77, 00145 Rome, Italy (pierluigi.vellucci@uniroma3.it).}} paper we prove stability results for Gabor frames and   bases   of $L^2(\R)$
that are relevant in  electronics and communication theory.
 A {\it Gabor system} in $L^2(\R)$ is a collection of functions  {$\G=\{e^{2\pi i b_n x}g(x-a_k)\}_{n,k\in\Z}$}, where  $g$ (the {\it window function}) is a fixed function in $L^2(\R)$ and  {$a_k,\ b_n\in\R$}.

If   ${\mathcal G} $ is  {\it regular  }, i.e., if {$(a_k, b_n)=(ak, bn)$} for some $a, b>0$,   we let  {$\G(g,a,b)=\{e^{2\pi  b i nx} g ( x-ak)\}_{n,k\in\Z} $}.

 Gabor systems have had a fundamental impact on the development of modern
time-frequency analysis and have been widely used in all branches of pure and applied sciences.

An important problem   is to determine general and verifiable conditions on the window
function $g$, the {\it  time  sampling  }   {$ \{a_k\}$} and the {\it  frequency  sampling   } {$\{b_n\} $} which imply that a Gabor system is a frame.  In the regular case many necessary and sufficient conditions  on $g$, $a$ and $b$ are known  (see  e.g. \cite{C1} and the references cited there).
An early article by Gr\"ochenig \cite{G}  provided some partial sufficient conditions
for the existence of irregular Gabor frames.  See also  \cite{FS} and \cite{BCHL} and the articles cited in these   papers.

 Given a regular Gabor frame $\G(g,a,b)$, it is   important   to determine  {\it stability bounds} $\delta_{n,k}>0$  so that each set  {$ \F=\{e^{2\pi i b \lambda_{n,k} x}g(x-a\mu_{n,k})\}_{n,k\in\Z}$  } is a frame
whenever  $   |\lambda_{n,k}-n|+|\mu_ {n,k}- k| <\delta_{n,k}$.
The main results of our paper concern the  stability of Gabor frames $\G(\rect^{(p)},a,b)$ where $\rect(x)= \chi_{[-\frac 12, \frac 12]}(x)$ is the characteristic function of the interval $[-\frac 12, \frac 12]$ and $\rect^{(p)}(x)=\rect*... *\rect(x) $ is the $p-$times iterated convolution of    $\rect(x)$. The function $\rect^{(p)}(x)$ is a piecewise polynomial function of degree $p-1$    and a prime example  of {\it B-spline} of order $p-1$.
See  \cite{Sch2},  \cite{PBP}, \cite{UAE}  and the references cited there.

 Our investigation  is  motivated by  the study of the timing jitter effect in    $p-$order hold (pOH) devices.  The  pOH  devices  are used to transform a sequence of impulses $\{q_n\}$ originating from a continuous-time signal $f(t)$ into a piecewise polynomial function  $f_p(t)$. The    impulses  are assumed to be evenly spaced, i.e., $q_n=f(Tn)$  for some $T>0$, but in the presence of timing jitter we have instead $q_n=f(T(n\pm \epsilon_n))$ for some $\epsilon_n>0$. It is natural to investigate whether $f(t)$ can be effectively reconstructed from the sequence $f(T(n\pm \epsilon_n))$.

It is proved  \cite{D}, (but see also \cite[chapt. 11]{He10}) that  the condition  $0< ab\leq 1$ is  necessary   for  a Gabor system $\G(g, a,b)$ to be a frame,  so we will always assume (often without saying)   that $ab\leq 1$.   When $\G(g, a,b)$ is a {\it Riesz basis}, i.e., it is  the image of an orthonormal basis of $L^2(\R)$ through a linear, invertible and bounded transformation, we have $ab=1$.

 We   consider   sets of coefficients $\{\mu_{n,k}  \}_{n,k \in\ZZ}\subset \R $, with
\begin{equation}\label{e-def-L} \fbox{$\displaystyle L_n=   \sup_{k\in\Z}|\mu_{n,k}-k| <1; \quad L=\sum_{n \in\ZZ}  L_n <\infty.
$}
\end{equation}
The assumption  $L_n<1$  is made  to   simplify the statement of our  result, but it is not necessary in the proofs.

We prove first a  stability result   for the frame $\G(\rect, a,b)$.

\begin{theorem}\label{T-stability-rect}
Assume $0<a\leq 1$, and   $\displaystyle
    4ab L
    <1.$
   The set
\begin{equation*}
\F=\left\{ e^{2\pi i b n t }\operatorname{\rect}\left( a\mu_{n,k}-t\right)\right\}_{k,n\in\mathbb Z}
\end{equation*}
is a frame of $L^2(\R)$ with bounds $A  \ge (1- 2( ab  L)^{\frac 12})^2$ and $B \leq (1+ 2(ab  L)^{\frac 12})^2$.
 In particular, $\F$ is a Riesz basis if $ab=1$ and $0\leq L<\frac 1{4 }$.

\end{theorem}

It is easy to verify that  $\G(\rect, a,b)$ is not a frame when  $a>1$;  indeed,
all  functions  in $\G(\rect, a,b)$
vanish on the intervals  {$\{(ak+\frac 12,\ a(k+1)-\frac 12)\}_{k\in\mathbb Z}$}, and so $\G(\rect, a,b)$ is not complete.

Observe also that if the coefficients $\mu_{n,k}$ in the definition of ${\mathcal F}$ are bounded below by a positive $\delta_k$ independent of $n$, then ${\mathcal F} $ is not a frame. For example, let $a=b=1$ and let $\mu_{n,k}=k$ if $k\ne 0$ and $\mu_{n,0}= d  $, where    $d  >0$ is fixed.   All functions in the set
 $ {\mathcal F}= \left\{ e^{2\pi i n t }\operatorname{\rect}\left( k-t\right)\right\}_{k\ne 0\atop{n\in\Z}} \cup
\left\{ e^{2\pi i n t }\operatorname{\rect}\left( d  -t\right)\right\}_{n\in\Z}
$
  vanish on the interval  $(-\frac 12, -d  )$  and so ${\mathcal F}$ is not  complete.

These considerations are not  new. See e.g.  \cite[Chapt 11]{He10} for similar observations.

\medskip

Theorem \ref{T-stabi-n-hold} below generalizes Theorem \ref{T-stability-rect}.

\begin{theorem}\label{T-stabi-n-hold}  Let $p\ge 1 $ be an integer; let   $ 0<a\leq 1$ when $p=1$  and $0<a<p $ when $p\ge 2$.
If
$  4ab L <   \left(\rect^{(p)} \left(\frac a2\right)\right)^2 $,
the set  $\F=\left\{e^{  2\pi i b nt  }\rect^{(p)} ( t- a\mu_{k,n}  )\right\}_{n,k\in\mathbb Z}$ is a  frame
with  bounds    $ A   \ge(1-  \frac{2 (ab   L)^{\frac 12}}{\rect^{(p)} \left(\frac a2\right)}    )^2$ and $B \leq (1+  \frac{2 (ab   L)^{\frac 12}}{\rect^{(p)} \left(\frac a2\right)}    )^2$.

 \end{theorem}

We also consider  window functions in the Sobolev space $H^1_0 (I)$, where $I$ is an interval of $\R$.
We recall that  $H^1 (I)$   is the space of  functions in   $L^2(I)$  whose   distributional derivative is also in $L^2(I)$ and that
$H^1_0(I)$ is the closure    of $ C^\infty_0(I)$ in $H^1 (I)$. We also recall that  functions in $H^1 (I) $ are continuous. See e.g. \cite{B}  for definitions and results on Sobolev spaces.

We  prove the following

\begin{theorem}\label{T-Sob}
Let $\psi\in H^1_0(-\frac p2, \frac p2)$, with $p>a>0$. Let $\psi'$ be the distributional derivative of $\psi$. Assume that     $0< m \leq \psi (x)\leq M$ in  $  [-\frac a2, \frac a2]$.  Then,
 \begin{itemize}
 \item $\G(\psi, a,b)$ is a frame with bounds $A\ge \frac{m^2}{b}$, $B\leq [\frac pa]\frac{M^2}{b}$.
 \item Let
$ \F=\left\{ e^{2\pi i b n t } {\psi}\left( a\mu_{n,k}-t\right)\right\}_{k,n\in\mathbb Z}
$
with $\{\mu_{n,k}\}_{n,k\in\Z}\subset \R$ as in \eqref{e-def-L}.
If  
$$\displaystyle
C^2= ba^2  ||\psi'||_2^2  \sum_{s\in\Z}\left[\frac pa+ L_s\right] L_s ^2
    < m^2 , $$
     ${\mathcal F}$ is a frame of $L^2(\R)$ with bounds $A \ge  \left( 1- \frac Cm\right)^2$  and
$B \leq  \left( 1+\frac Cm\right)^2$
   \end{itemize}
\end{theorem}

  The following is   a  short (and most likely incomplete) survey of results related to our work.
 {  In   \cite{FG}   the stability of Gabor frames is   tested under  perturbation
of the lattice constants $a$ and $b$.}
 In   \cite{Favier95}  the stability of  Gabor  frames and bases  $\G(\phi, a,b) $
 under some perturbation of the sampling sequence  {$\{ak\}_{k\in\Z}$} is discussed.   See also  \cite{C}.      The  assumptions of  Theorem 16 in \cite{Favier95}
  do  not apply to the frames that we have considered in this paper.
  Some results in \cite {Favier95}   have been improved by
  W. Sun and X. Zhou  in \cite{Sun01}. {  A proof of the main theorem in \cite{Sun01} is in \cite[Theorem 15.4.3]{C1}.  In this theorem  the  window function $g$  is continuously differentiable while  in our Theorem \ref{T-Sob} we consider functions in $H^1(I)$. }

 The stability of Gabor frames with irregular sampling points is considered in
  \cite{FS}, \cite{SZ2}.
{ It is also worth mentioning that in \cite[Theorem 15.4.1]{C1}    and in  \cite[Corollary 15.4.2]{C1}, the stability of Gabor frame under  perturbation of the  window function $g$ is discussed.}

 We prove Theorems \ref{T-stability-rect} -- \ref{T-Sob} in Section 3. In Section 2 we recall definitions and preliminary results and we prove some useful Lemmas. In Section 4 we prove   corollaries and   a generalization of  Theorem \ref{T-stability-rect}  in dimension $d>1$.
 We have described the $p$-order hold models in the Appendix.

\section{Preliminary}\label{S-prelim}

 We refer the reader to the  the excellent textbooks \cite{He10} and \cite{You} for the definitions of frame and Riesz  bases and preliminary results. See also \cite{C1}, \cite{G2}.

We recall from  \cite[Chapt. 11]{He10} that a Gabor system
   $\G(g,a,b)$   is not  a frame if $ab> 1$. If $\G(g, a,b)$ is a frame  with bounds $A$ and $B$, then \begin{equation}\label{e-cond-AB}
Ab\leq \sum_{k\in\Z} |g(x+ak)|^2\leq B b
\end{equation}
holds almost everywhere (a.e.)  in $\R $.

When  $g$  is supported   in an interval of measure  $\ge a$ and    $0< ab\leq 1$,  the set      $\G(g, a, b)$ is a   frame of $L^2(\R)$ with constants $A$ and $B$ if  and only if  \eqref{e-cond-AB} is satisfied;  if $\G(g, a, b)$ is a frame and $ab =
1$, then it is a Riesz basis.
 It follows  from \eqref{e-cond-AB} that   if $g$ is   continuous  with compact support in $\R$ and $\G(g, a, b)$ is   a frame,   the translates of $g$ must overlap. Indeed,  it is easy to verify that
if  $g\in C(\R)$ is supported in the interval $(x_0, x_0+\frac 1b)$ for some $x_0\in\R$,    and if $a\leq \frac 1b$,
 the inequality in  \eqref{e-cond-AB}  can only hold  with $A=0$.

\medskip
\noindent
{\it Example.}  Let $p\ge 1$ and  $\G_p=\G(\rect^{(p)}, a, b )$; when $p\ge 2$,
the functions  $\rect^{(p)}$  are continuous    and  satisfy  \eqref{e-cond-AB}    with a   constant $>0$ only when    $0<a<p$ and  $0< b\leq \frac 1a$.
The set $\G_p $ is a  Riesz basis when  $b=\frac{1}{a}$;
      because the  functions  $\rect^{(p)}(x-ka)$  overlap,
$\G_p$ can never be   orthogonal.
On the other hand, the set  $\G_1=\G(\rect, a, b )$  is a frame if   $0\leq a\leq 1 $  and $0<b\leq \frac 1a$. When $a=b=1$, $\G_1$ is  orthonormal.

\medskip

  %\subsection{Stability results}\label{S-stab-Gabor}

 Let us recall some stability results that we  need for our proofs. The following    is    Theorem 2.3 in \cite{Sun99}.

\begin{theorem}\label{T-SunZ} Let $\phi\in L^2(\R)$ be supported in an interval of length $\frac 1b$.  Assume that   $\G(\phi,a,b)$ is a frame in $L^2(\R )$ with bounds $A$ and $B$. Let $\{\lambda_n\}_{n\in\Z } \subset\R $ be  such that
$\displaystyle
 {|\lambda_n-n|}\leq L <\frac 14\quad \mbox{ for every $n\in\Z$. }
$ The set $\{e^{2\pi i b\lambda_n t}\phi(t-ak)\}_{n,k\in\Z}$ is a frame in $L^2(\R )$ with bounds $A'\ge A( \cos( \pi L)-\sin( \pi L))$ and $B'\leq B(2-\cos( \pi L)+ \sin( \pi L))$.
\end{theorem}

The following theorem is a consequence of \cite[Theorem 1]{C2} (see also  \cite{Sun01}).
\begin{theorem}\label{T-stab-frames} Let $\{x_n\}_{n\in\Z}$ be a frame for a Hilbert space $H$ with bounds $A$ and  $B$. Let $\{y_n\}_{n \in\Z}\subset H$ be such that the inequality
$  \left\|\sum_ n a_n (x_n-y_n)\right\|^2\leq C  \sum_n |a_n|^2
$
 is valid  with a constant $0< C<A$.
Then, $\{y_n\}_{n\in\Z}$ is a frame with bounds $A'\ge (1-(\mbox{$\frac C A$})^{\frac 12})^2 $ and
$B'\leq (1+(\mbox{$\frac C A$})^{\frac 12})^2 $. If $\{x_n\}_{n \in\Z}$ is a Riesz basis, then
 $\{y_n\}_{n\in\Z}$
is also a Riesz basis.

%The upper frame bound was (1+(C/B)^{\frac 12})^2
\end{theorem}

We conclude this section with the following observation. Let  $\alpha,\, \beta, \, T\in\R$ with $T>   0$; let $\psi\in L^2(\R)$ and let $\hat\psi(x)= \int_\R\psi(t)e^{-2\pi i x t}dt$ be the  Fourier transform of $\psi$. The Fourier transform of $ t\to e^{2\pi i \beta   \frac t T}\psi( \frac t T-\alpha )$ is $x\to T e^{2\pi i \alpha  (  Tx-\beta)}\hat \psi  (\beta - Tx)$. In particular, the  Fourier transform of
$ t\to e^{\frac{2\pi i t \beta }T}\operatorname{rect}^{(p)}\left(\frac{t}{T}-\alpha\right)$ is $x\to Te^{2\pi i \alpha (\beta-Tx) }(\operatorname{sinc} (\beta-Tx))^p$, where $\sinc (t)=\frac{\sin(\pi t)}{\pi t}$ when $t\ne 0$ and $\sinc( 0)=1$.
 If $ \{f_j\}_{j\in\Z}$ is a frame in $L^2(\R)$, by  Plancherel's theorem   the set $ \{\hat f_j\}_{j\in\Z}$ is also a frame in $L^2(\R)$ with the same frame bounds (see e.g. \cite[Prop. 11.2.5]{C}).
 In particular,    the  Fourier transform  maps  $\G(\rect ^{(p)}, a, b)$ into  $\G(e^{2\pi i ab}\sinc^p,\  b  -a )$, and  stability results for  frames  $\G(\rect ^{(p)}, a, b)$ yield stability results for frames $\G(\sinc ^{p}, b,-a)$ and vice versa. A sample result is Corollary \ref{C-stabi-sinc} in Section 4.

 \subsection{Two useful lemmas}

{   We recall that for every $p\ge 1$, $\rect^{(p)}(t)  $ is the $p-$fold convolution of $\rect(t)$.  Thus, for every $p\ge 1$,
\begin{equation}\label{eq:4}
\rect^{(p+1)}(x)=\rect(t)*\rect^{(p)}(t)= \int_{x-\frac 12}^{x+\frac 12}\rect^{(p)}(  t)dt.
\end{equation}
Using the identity  \eqref{eq:4} we can easily prove that  $\rect^{p}(t)$    is    supported in the interval $[-\frac p2, \frac p2]$, that   $0\leq \rect^{(p)}(t)\leq 1  $ for every $t\in\R $ and  that   $\rect^{(p)}$   is   continuous when $p\ge 2$ and is differentiable when $p\ge 3$. It is not too difficult to verify that  $\rect^{(p)}(x)$ is   increasing   in      $(- \infty,\ 0)$      and is decreasing in  $(0, \infty)$.

In \cite{Sch1}   the following identity has been proved for every $p\ge 2$.
\begin{equation}\label{e-form-schon}
\rect^{(p)} (x)=\frac{1}{(p-1)!}\sum_{j=0}^p (-1)^j \binom{p}{j}\left(x+\frac{p}{2}-j\right)_{+}^{p-1}\end{equation}
where
$ x_{+}=\left\{
    \begin{array}{ll}
      x, & \hbox{if}\ x\geq0; \\
      0, & \hbox{if}\ x<0.
    \end{array}
  \right.
$
See also \cite{Sch2}. This formula shows that  $\rect^{(p)}(x)$  is a   polynomial of degree $p-1$
in intervals   $\Lambda_i^p(x)= [-\frac{p}{2}+i-1,-\frac{p}{2}+i ] $, with
  $1\leq i\leq p$.

We use the notation  $[x]$ to denote the integer part of a real number $x$.  When $x\ge 0$,  $[x]$ is the  integer $n\ge 0 $ that satisfies   $n\leq x< n+1$.

 \begin{lemma}\label{L-Framec-rect}  For every $p\ge 1$, the  optimal  frame bounds  $A_p$ and $B_p$ of   $\G(\rect^{(p)}, a,b)$
satisfy $B_p \leq
\frac 1b
[\frac pa]$ and $A_p\ge
\frac 1b
(\rect^{(p)}(\frac a2))^2$.
 \end{lemma}

\noindent
{\it Remark.} The exact frame constants   of    $\G(\rect^{(p)}, 1,1)$ are evaluated in
 \cite{M} (see also \cite{SR}). They are   $B_p=1$ and $A_p =\frac{2^{2p+1}}{\pi^{2p}}\sum_{\nu=0}^\infty \frac{ 1}{(2\nu+1)^{2p }}$.
The expression  $K_m=\frac 2\pi \sum_{\nu=0}^\infty \frac{(-1)^{\nu(m+1)}}{(2\nu+1)^{m+1}}$ is known as the {\it Krein-Favard constant}.
   \begin{proof}[Proof of Lemma \ref{L-Framec-rect}]
    Let $ \displaystyle
 S_p(a,x)=\sum_{n\in\Z} (\rect^{(p)} (x-an))^2.
$
 By \eqref{e-cond-AB},
 $$A = \frac 1b\inf_{x\in\R} S_p(a,x)\ \ \text{and} \ \ B = \frac 1b\sup_{x\in\R} S_p(a,x)$$
 are   frame bounds of  $\G(\rect^{(p)}, a,b)$.  Note that  $ S_p(a,x)$ is periodic with  period $a$, i.e., $S_p(a,x)=S_p(a, a+x)$.  Thus,  $bA $ and $bB $ equal the minimum and maximum  of $S_p(a,x)$ in the interval $[-\frac a2, \frac a2]$.

The maximum of $S_p(a,x)$  is easy to evaluate: the functions  $\rect^{(p)}(t-na)$ are supported in  the interval $J_n =[na-\frac p2, na+\frac p2]$  and    each  point
in the interval $[-\frac a2, \frac a2]$  belongs to at most
  $[\frac pa] $   overlapping $J_n$'s. Since  $\rect^{(p)}(t)\leq 1$ we have  that  $S_p(a,x)\leq [\frac pa]$.

To estimate  the minimum of  $S_p(a,x)$   we   observe that
   $$
   \min_{x\in\R} S_p(a,x) = \min_{|x|\leq\frac a2} S_p(a,x)\ge \min_{|x|\leq\frac a2} (\rect^{(p)} (x))^2.
   $$
Recalling that  $(\rect^{(p)} (x))^2$ is even, and is increasing when $x<0$ and   decreasing when $x>0$,  we can see at once that $ \min_{|x|\leq\frac a2} (\rect^{(p)} (x))^2=  (\rect^{(p)} (\frac a2))^2
   $ as required.
\end{proof}

The following lemma will be used to prove  Theorems  \ref{T-stability-rect}  and  \ref{T-stabi-n-hold}.
 \begin{lemma}\label{L-stab-Vo}   Let
   $\{\mu_k\}_{k\in\Z}\subset\R$  be  such that {$\ell=
 \sup_{k\in\Z} |\mu_k-k|  < \infty$}.
Then, for every finite sequence  {$\{\alpha_k\}\subset\ell^2$} we have that
   \begin{equation}\label{e-cond-PW}\left\Vert\sum_k \alpha_k(\rect(t-ak)-\rect(t-a\mu_k)\right\Vert_2^2  < 2a(z+1)\ell \sum_k  |\alpha_k|^2
\end{equation}
where $z=0$ when  $0<\ell\leq\frac{a-1}{2a}$ and $z= [\ell]+1$   in all other cases.

\end{lemma}

\medskip
\begin{proof} We can assume   $\sum_k  |\alpha_k|^2=1$.
Let
 {\begin{equation}\label{e-sum-In}
f(t)=\sum_k \alpha_k\left(\rect(t-ak)-\rect(t-a\mu_k)\right)= \sum_k \alpha_k \chi_{I_k}(t),
\end{equation}}
where $I_j $ denotes the support  of $  \rect (t-aj)-\rect (t-a\mu_j)$.
  When $ \mu_j\ne   j$, $I_j$ is union of two intervals that we  denote with $I_j^+$ and $I_j^-$. When $ \mu_j >   j$ we let
 $$ I_j^-=(aj-\frac 12, \ a\mu_j-\frac 12),\quad  I_j^+=(aj+\frac 12, \ a\mu_j+\frac 12).
  $$
We   use (improperly)  the same notation  also when  $\mu_j<j$.
We can write \eqref {e-sum-In} as:
$$ f(t)=  \sum_k \alpha_k \chi_{I_k^+}(t) +\sum_k \alpha_k \chi_{I_k^-}(t).
$$
 If    $a>1$  and  if  $a-1$,  the  measure of the ``gap'' between   the supports of  the     {$\rect(t-ak)$}, is larger than  $2a\ell$,  the $I_j$'s do not intersect (see Figure 1).  From  \eqref {e-sum-In} follows that
 $
||f||_2^2=\sum_{k\in\Z} |\alpha_k|^2 |I_k| \leq 2a\ell
$
 whenever $a>1$ and   $\ell<\frac {a-1}{2a}$.

     \begin{figure}[h!]
  \begin{center}
     \begin{tikzpicture}
      \draw [black ] (-1,0) --(9.5,0);
      \draw [black, dashed] (0,0) node[black,    below] {\tiny{$aj-\frac 12$}} --(2.5,0)
      node[black,  below ] {\tiny{$a j+\frac 12$}}--(2.5, 1.5 )--(0, 1.5)--(0,0);
      %\draw [black] (2.5, 1.5 )--(0, 1.5);
      \draw [black,fill]  (0,0) circle [radius=0.07];
      \draw [black,fill]  (2.5,0) circle [radius=0.07];
      \draw [black, dashed] (5,0) node[black,    below] {\tiny{$a(j+1)-\frac 12$}} --(7.5,0)
      node[black,  below ] {\tiny{$a(j+1)+\frac 12$}}--(7.5, 1.5 )--(5, 1.5)--(5,0);
      \draw [black, dashed] (7.5, 1.5 )--(5, 1.5);
 \draw [black,fill]  (5,0) circle [radius=0.07];
      \draw [black,fill]  (7.5,0) circle [radius=0.07];
    \draw [black ] (1,0)--(1,1.4)  --(3.5,1.4)--(3.5,0);
      \draw[black] (3.5, -.3) node[black,  below ] {\tiny{$a \mu_{j } +\frac 12$}};
      \draw[black] (1, -.3) node[black,  below ] {\tiny{$a \mu_{j } -\frac 12$}};
 \draw [black,fill]  (1,0) circle [radius=0.07];
      \draw [black,fill]  (3.5,0) circle [radius=0.07];
    \draw [black ] (6,0)--(6,1.4)  --(8.5,1.4)--(8.5,0);
      \draw[black] (8.5, -.3) node[black,  below ] {\tiny{$a \mu_{j+1} +\frac 12$}};
      \draw[black] (6, -.3) node[black,  below ] {\tiny{$a \mu_{j+1} -\frac 12$}};
 \draw [black,fill]  (6,0) circle [radius=0.07];
      \draw [black,fill]  (8.5,0) circle [radius=0.07];
      \draw [black ] (6,0)--(6,1);
      \draw [black, ultra thick ] (0,0)--(1,0);
      \draw [black, ultra thick ] (2.5,0)--(3.5,0);
      \draw [black, ultra thick ] (5,0)--(6,0);
      \draw [black, ultra thick ] (7.5,0)--(8.5,0);
      \draw  (.5, 0) node[black, above] {$\tiny I_j^-$};
       \draw  (3, 0) node[black, above] {$\tiny I_j^+$};
       \draw  (5.5, 0) node[black, above] {$ {\tiny I_ {j+1}^-}$ };
       \draw  (8, 0) node[black, above] {${\tiny I_ {j+1}^+}$ };
       \end{tikzpicture}
    \caption{  }
   \end{center}
 \end{figure}

In all other cases,  the intervals $I_j^{\pm}$  may  intersect. Let  {$I_h$} be fixed, and let  {$z=z(h)$} be the maximum number of intervals $I_{j} $, with  {$j\ne h$}, that intersect  {$I_h$}. The sum in \eqref{e-sum-In} has at most $z+1$ terms because there are at most  $z+1$  functions $\chi_{I_k}$  that overlap at each point.
The elementary inequality
\begin{equation}\label{e-elem} \left(x_1+\dots+x_m\right)^2\leq m \left(x_1^2+\dots+x_m^2\right)\end{equation}
which is valid for every $x_1$, $\dots$, $x_m\in\mathbb R$, yields  $|f(x)|^2 \leq (z+1) \sum_k |\alpha_k|^2  \chi_{I_k}(x)$, and
$ ||f||_2^2\leq (z+1)\sup_k |I_k|=2\ell(z+1).$

It remains to determine $z$.  Clearly    $z$ is maximum  when   $\mu_j= j \pm \ell$, so we assume that this is the case.
For simplicity we let $I_h=I_0 $, with $I_0^+=(  \frac 12, \ a\ell+\frac 12)$ and $I_0^-= (  -\frac 12, \ a\ell-\frac 12)$, and $\mu_j=j+\ell$ (the case $\mu_j=j-\ell$ is similar).
%We count the intervals $I_j^{\pm}$, with $j\ne 0$, that intersect $I_0^+$.

 The interval   $I_j^+ =(aj+\frac 12,\  a(j+ \ell)+\frac 12)$ intersects $I_0^+$  if either:   (a) $\frac 12<aj+\frac 12 <a\ell+\frac 12$, or  (b) $\frac 12<a(j + \ell  )+\frac 12 <a\ell+\frac 12$.  The inequality a) is equivalent to $0<j<\ell$;  the inequality b) is equivalent to $0<j+ \ell<\ell$, which is satisfied when $-\ell < j<0$; since  a) and b) cannot be verified
 simultaneously, there are    at most $ [\ell] $  integers $j$  for which $I_j^+$ intersects $I_0^+$.

 The interval  $I_j^-= (  aj-\frac 12, \ a(j+ \ell)-\frac 12)$  intersects  $I_0^+$ if either:  {(a) $\frac 12<aj-\frac 12 <a\ell+\frac 12$, or (b) $\frac 12<a ( j+ \ell)-\frac 12 <a\ell+\frac 12$.}
Equivalently,
 $ \mbox{a)}  \quad
  \frac 1a < j  < \ell+\frac 1a, \quad \mbox{ or $ b)$}\quad
 \frac 1a - \ell <j <  \frac 1a
 $
and so there are   at most $ [\ell]+1  $ integers  $j$  for which $I_j^-$ intersects $I_0^+$. %\footnote{ma qui ci possono essere [L]+1  interi. Per esempio, se  $1/a= 0.7$ e $L=1.4$ abbiamo $.7< j<2.1$ e quindi $j$ puo' essere  anche $j+1$}

To summarize: we have $z =0$ when $a>1$ and $\ell <\frac {a-1}{2a}$, and $z \leq 2([\ell]+1) $   in all other cases, and
the proof of   Lemma \ref{L-stab-Vo} is concluded.
 \end{proof}

\section{Proofs of Theorems \ref{T-stability-rect}, \ref{T-stabi-n-hold} and \ref{T-Sob} }\label{sect-proofs}

\begin{proof}[Proof of Theorem \ref{T-stability-rect}]
By  Lemma \ref{L-Framec-rect}, the lower  frame bound  of $\G(\rect, a,b)$ is  $A\ge \frac 1{b}$. Fix a finite set of coefficients
 $\{\alpha_{n,k}\}\subset\C$ with $\sum_{n,k}  |\alpha_{n,k}|^2=1$; let
$$ f(t)=\sum_{n,k\in\ZZ} \alpha_{n,k} e^{2\pi i b nt }\left(\operatorname{rect}\left(  t-ak\right)-\operatorname{rect}\left( t-a\mu_{k,n}\right)  \right) .$$
If we show that
$||f||_2   < \frac 1b$  for every set of  coefficients  $\{\alpha_{n,k}\}$,  by
 Theorem  \ref{T-stab-frames}   the set $\F$ in Theorem \ref{T-stability-rect} is a frame in $L^2(\R)$  and   a Riesz basis when $ab=1$.

We let
\begin{equation*}
  f_n(t)  = \sum_{k\in\ZZ} \alpha_{n,k}  \left(\operatorname{rect}\left(  t-ak\right)-\operatorname{rect}\left( t-a\mu_{k,n}\right)  \right) \ \ \text{and} \ \
  f(t)  = \sum_{n\in\Z}e^{2\pi i b nt } f_n(t)\, .
\end{equation*}
By  the triangle inequality, $||f||_2\leq \sum_{n\in\Z} ||f_n|| _2$;
we can apply   Lemma \ref{L-stab-Vo}  with $z=1$  because we have assumed in \eqref{e-def-L} that $L_n=\sup_k|\mu_{n,k}-k|<1$. We gather    $||f_n||_2^2 \leq  4aL_n   \sum_k |\alpha_{n,k}|^2  $ and
$$
||f||_2\leq \sum_{n\in\Z} \sqrt{4aL_n}\left(\sum_k |\alpha_{n,k}|^2 \right)^{\frac 12}.
$$
By  H\"older's inequality,  $\displaystyle
||f||_2^2\leq 4a \sum_{n\in\Z}L_n= 4aL.
$
By Theorem \ref{T-stab-frames},  the set  {$\F$}   in  Theorem \ref{T-stability-rect} is a frame  with constants
$A'\ge  (1- 2(ab L)^{\frac 12})^2$ and $B'\leq \frac 1b(1+  2(ab L)^{\frac 12})^2$ whenever
$
   4a L <\frac 1b
   $,
     as required.
 \end{proof}

\begin{proof}[Proof of Theorem \ref{T-stabi-n-hold}]
%Consider first the case $ab=1$.

Let $\{\alpha_{n,k}\}\subset\C$ be a finite sequence such that $\sum_{n,k}  |\alpha_{n,k}|^2=1$.
 Let \begin{equation}\label{e-def-fp}f_p(t)=\sum_{n,k\in\ZZ} \alpha_{n,k} e^{  2\pi i b nt  }\left(\rect^{(p)}(  t-ak)- \rect^{(p)} ( t- a\mu_{k,n}  )\right).
  \end{equation}
  We  prove that, for every $p\ge 1$,
\begin{equation}\label{e-induct}
||f_p||_2^2\leq  4aL
\end{equation}
 where $L=\sum_{n\in\Z}\sup_{h\in\Z}|\mu_{n,h}-h| $ is as in the proof of Theorem \ref{T-stability-rect}.
 By Lemma \ref{L-Framec-rect}, the frame constants of $\G(\rect^{(p)}, a,b)$ are
$A_p\ge \frac 1b (\rect^{(p)} (\frac a2))^2$ and $B_p \leq \frac 1b[p/a]$ and
by assumption
$4aL <\frac 1b (\rect^{(p)} (\frac a2))^2 \leq A_p$. We can use
   Theorem  \ref{T-stab-frames} to conclude the proof of the theorem.

We prove \eqref{e-induct} by induction    on $p$.
The $L^2(\R)$ norm of the function
  $ f_{1 }(t) $
  has been estimated in   Theorem \ref{T-stability-rect} and we have proved \eqref{e-induct}.
 Assume    that \eqref{e-induct} is satisfied by $f_{p-1}$,  with   $p\ge 2$.
Recalling that  $\rect^{(k)}(x)=\rect*\rect^{(k-1)}(x)=  \int_\R \rect(x-y)\rect^{(k-1)}(y) dy$,     \eqref{e-def-fp} yields
 \begin{align*}
f_p(t) &= \int_{-\infty}^\infty\!\!\!  \rect(t-y)  \! {\sum_{n,k\in\ZZ} \!\!\! \alpha_{n,k} e^{ 2 \pi ibnt}\! \left(\rect^{(p-1)} ( y-ak )\!-\! \rect^{(p-1)} ( y-a\mu_{k,n}  )\right)dy}
\\ &
 = \int_{-\infty}^\infty  \rect(t-y)  e^{ 2 \pi ibn(t-y)}  f_{p-1}(y)dy
= \left(  {e^{ 2 \pi ibn t  } \rect t}\right)* f_{p-1 }(t).
\end{align*}
   By  Young's inequality for convolution and \eqref{e-induct},
$$
||f_p||_2\leq ||\rect ||_1||f_{p-1}||_2 \leq  \left(4a  L \right)^{\frac 12}\leq \frac 1{b^{\frac 12}} (\rect^{(p)} (\mbox{$\frac a2$}))
$$
as required.
\end{proof}

\medskip
\begin{proof}[Proof  of  Theorem \ref{T-Sob} ]
Let $S(x)=\sum_{k\in\Z} |\psi(x+ak)|^2$. We can argue as in the proof of  Lemma \ref{L-Framec-rect} to show that  the frame constants of $\G(\psi, a, b)$ are $A=\frac 1b \min _{x\in [-\frac a2, \frac a2]} S(x)$ and $B =\frac 1b \max _{x\in [-\frac a2, \frac a2]} S(x)$.
Clearly,
$$A\ge \frac 1b\min _{x\in [-\frac a2, \frac a2]} |\psi(x)|^2 \ge \frac{m^2}{b} \, .$$
Since each point in $[-\frac a2, \frac a2]$ belongs to  the support of at most  $[\frac pa]$   functions in the sum above, we have that
$B\leq \frac 1b [\frac pa]\max_{x\in [-\frac a2, \frac a2]} |\psi(x)|^2 \leq \frac 1b[\frac pa] M^2$, as required.

\medskip
Fix a finite set of coefficients
 $\{\alpha_{n,k}\}\subset\C$ with $\sum_{n,k}  |\alpha_{n,k}|^2=1$; let
$$ f(t)=\sum_{n,k\in\ZZ} \alpha_{n,k} e^{2\pi i b nt }\left(\psi\left(  t-ak\right)-\psi\left( t-a\mu_{k,n}\right)  \right) .$$
If we show that
$||f||_2^2   < \frac{C^2}{b} <A $,  by Theorem \ref{T-SunZ} we can conclude that $\F$ is a frame in $L^2(\R)$.

We argue as in the proof of Theorem \ref{T-stability-rect}.
We let  $  f_n(t)= \sum_{k\in\ZZ} \alpha_{n,k}  ( \psi(  t-ak)- \psi( t-a\mu_{k,n})
$  and $
f(t)= \sum_{n\in\Z}e^{2\pi i b nt } f_n(t)
$.
By  the triangle inequality, $||f||_2\leq \sum_{n\in\Z} ||f_n|| _2$;

To estimate  the norm of  $f_n$ we argue as in  Lemma \ref{L-stab-Vo}.
Fix $n\in\Z$; let
   $\{\mu_{n,k}\}_{k\in\Z}\subset\R$, with   $L_n=
 \sup_{k\in\Z} |\mu_{n,k}-k|  $.  Let $\{\alpha_{n,k}\}\subset \C$ be a finite sequence. We show that
\begin{equation}\label{1}\left\Vert\sum_k \alpha_{m,k}(\psi(t-ak)-\psi(t-a\mu_k)\right\Vert_2^2   \leq \left[\frac pa+ L_n\right] (aL_n)^2 ||\psi'||_2^2.\end{equation}
Fix $k\in\Z$; if  $\mu_{n,k}>k$  (resp.  $\mu_{n,k}<k$) the support of the function $g_k(t)=\psi(t )-\psi(t- a(\mu_{n,k}-k))$ is in the interval $[ -\frac p2, \, aL_n+\frac p2]$   (resp. $[ -\frac p2-aL_n, \,  \frac p2]$)  and so   there are at most $[\frac pa+ L_n]$ supports of functions $g_k(t-ak)$ that overlap at each point.  In view of  the elementary inequality \eqref{e-elem} we obtain

\begin{align*}
 \left\Vert\sum_k \alpha_{m,k}(\psi(t-ak)-\psi(t-a\mu_{n,k}))\right\Vert_2^2 &=\left\Vert\sum_k \alpha_{n,k}g_k(t-ak)\right\Vert_2^2
 \\ & \leq
\left[\frac pa+ L_n\right]  \sum_k \alpha_{n,k}^2  \Vert  \psi(.-ak)-\psi(.-a\mu_{n,k}))\Vert_2^2
\\  & =
\left[\frac pa+ L_n\right]  \sum_k \alpha_{n,k}^2  \left\Vert  \int_{t-ak}^{t-a\mu_{n,k}}\psi'(\xi)d\xi \right\Vert_2^2
\\  & \leq
\left[\frac pa+ L_n\right]  \sum_k \alpha_{n,k}^2  \left(\int_{ ak}^{ a\mu_{n,k}} \left\Vert \psi'(.-s)\right\Vert_2 ds \right)^2
\\
& \leq \left[\frac pa+ L_n\right]  \sum_k \alpha_{n,k}^2 (aL_n)^2 ||\psi'||_2^2   \\ & = \left[\frac pa+ L_n\right]  (aL_n)^2 ||\psi'||_2^2 \sum_j|\alpha_{n,j}|^2.
\end{align*}
We have used the change of variables  $\xi=t-s$ in the  integral on the third line   and Minkoswky's inequality. The proof of Theorem \ref{T-Sob} is concluded.

\end{proof}

\section{Corollaries and generalizations}

 In this section we prove corollaries  and generalizations of Theorems \ref{T-stability-rect} and \ref{T-stabi-n-hold}.
 %As in previous sections, we  consider Gabor frames   $\G(g, a,b)$ with  $0< ab\leq 1$.    %All frames in the corollaries below are Riesz bases when $ab=1$.
%
We start with two  corollaries of Theorem \ref{T-stabi-n-hold}.

\begin{corollary}\label{C-2par-rect}
Let $p\ge 1$ be an integer. Let  $ 0<a< p$ when $p>1$ and $0\leq a\leq 1$ when $p=1$.
Let  $\{\lambda_n\}_{n\in\Z}\subset\R$  that satisfies
$|\lambda_n-n|\leq \ell<\frac 14$; let  $\{\mu_{n,k}  \}_{n,k \in\ZZ}\subset \R$  be as in \eqref{e-def-L}.   If
$$
  \   4abL< \left[ \rect^{(p)}(a/2)\right]^2(\cos(\pi \ell)-\sin(\pi \ell)),$$
the set $ \{e^{2\pi i \lambda_n bt}\rect^{(p)}(t-a\mu_{n,k})\}_{n,k\in\Z}$ is a frame with constants
\begin{equation*}
A''_p = \left(1- \frac{2}{\rect^{(p)}(a/2)} (\frac{abL}{\cos(\pi \ell)-\sin(\pi \ell)})^{\frac 12}\right)^2
\end{equation*}
and
\begin{equation*}
B''_p= \left(1+\frac{2}{\rect^{(p)}(a/2)} (\frac{abL}{\cos(\pi \ell)-\sin(\pi \ell)})^{\frac 12}\right)^2\, .
\end{equation*}
\end{corollary}
\begin{proof} Let    $c= 1- \cos(\pi \ell)+\sin(\pi \ell) $.
Recalling that  the frame constants  of $\G(\rect^{(p)}, a,b)$ are $A_p\ge \frac 1b(\rect^{(p)}(a/2))^2$ and $B_p \leq  \frac 1b[\frac pa]$, by Theorem \ref{T-SunZ}
the set $ \{e^{2\pi i \lambda_n bt}\rect^{(p)}(t-ak)\}_{n,k\in\Z}$   is a frame with   constants
$A'_p\ge \frac 1b(\rect^{(p)}(a/2))^2 (1-c)$ and $B'_p\leq \frac 1b[\frac pa](1+c)$.

The proof of Theorem \ref{T-stabi-n-hold} can be repeated line by line for the
set {$\F= \{e^{2\pi i \lambda_n bt}\rect(t-a\mu_{n,k})\}_{n,k\in\Z}$}, with $\lambda_n$ in place of $n$.  We  conclude that if $4aL< A'_p   $, the set
$\{e^{2\pi i \lambda_n bt}\rect(t-a\mu_{n,k})\}_{n,k\in\Z}$
is a frame with   constants
$A''_p = (1- \frac{2}{\rect^{(p)}(a/2)} (\frac{abL}{1-c})^{\frac 12})^2 $ and $B''= (1+\frac{2}{\rect^{(p)}(a/2)} (\frac{abL}{1-c})^{\frac 12})^2$.
\end{proof}

 Plancherel's theorem and Corollary \ref{C-2par-rect} yield the following

\begin{corollary}\label{C-stabi-sinc}
With the notations and assumptions of    Corollary \ref{C-2par-rect}, the set
$$ \left\{ e^{-2\pi i a\mu_{n,k} t }\operatorname{sinc}^p\left(t- b\lambda_n \right)\right\}_{k,n\in\mathbb Z}
$$ is frame  of $L^2(\R)$ with constants $A''_p$ and $B''_p$.
\end{corollary}
\medskip

We  prove  a multi-dimensional version  of Theorem   \ref{T-stability-rect}.
 We will use the following notation:
for any two vectors $\textbf{v}=(v_1,\dots,v_d)$ and $\textbf{w}=(w_1,\dots,w_d)\in \mathbb \R^d$, we let  $\langle \textbf{v},\ \textbf{w} \rangle =\sum_{j=1}^d v_jw_j$ and $|\bf w| = \sqrt{\langle \bf w, \ \bf w\rangle} $  be the usual scalar product  and norm in $\mathbb R^d$.
 We      denote with $\textbf{v}\cdot \textbf{w}=(v_1\, w_1,\dots,c_d\, w_d)$ the Hadamard (componentwise) product of the vectors $\textbf{v}$ and $ \textbf{w}$.  We let  $\operatorname{rect}\left( \textbf{x}\right)=\rect(x_1, ...,\, x_d)$ be the product of the functions   {$\rect(x_1)$, \dots ,\, $\rect(x_d)$.

\medskip
\begin{theorem}
\label{T-d-dim}
Let $\textbf{a}=(a_1,\dots,a_d)$ and $\textbf{b}=(b_1,\dots,b_d)$, where $0<a_j\leq 1$ and  {$0< a_jb_j\leq 1$ for $j=1,\dots,d$}.  If $\{\boldsymbol\mu_{\textbf{k},\textbf{n}}\}_{\textbf{k},\textbf{n}\in\mathbb Z^d}\subset \R^d$, with $\boldsymbol\mu_{\textbf{k},\textbf{n}} =\left(\mu_{k_1,n_1},\dots,\mu_{k_d,n_d}\right)$, is a set of  vectorial coefficients that satisfy
\begin{equation}\label{eq:pv}
{\mathcal L}_j=  \sum_{n_j\in\ZZ} \sup_{h_j\in\Z}\left|\mu_{n_j,h_j}-h_j\right|   <  {\frac {1}{4 a_j b_j}}, \ \ \ j=1,\dots,d,
\end{equation}
the set
$\B=\left\{ e^{2\pi i \langle \textbf{b}\cdot\textbf{n}, \textbf{t} \rangle}\operatorname{rect}\left(\textbf{a}\cdot\boldsymbol\mu_{\textbf{n},\textbf{k}}-\textbf{t}\right)\right\}_{\textbf{n},\textbf{k} \in\ZZ^d}$
is a frame for  $L^2(\R^d)$ with constants $A\ge \prod_{j=1}^d(1- 2( a_j b_j  {\mathcal L}_j)^{\frac 12})^2$ and  {$B\leq\prod_{j=1}^d(1+ 2( a_j b_j  {\mathcal L}_j)^{\frac 12})^2$}.
\end{theorem}
\begin{proof}
We prove the theorem  only for $d = 2$ (the proof for $d >2$  is similar).

To prove that $\B$ is a frame it is enough to show that for every $f\in L^2(\R^2)$ with  $||f||_{L^2(\R^2)}=1$, we have that
$$A\leq\!\!\!\!  \sum_{(n_j,k_j)\in\Z^2}\!\!    \left(\int_{\R^2}\!\!\!\! f(t_1, t_2)e^{-2\pi i (b_1 n_1 t_1+b_2 n_2t_2) }\!\! \prod_{j=1,2}\rect(t_j-a_j\mu_{n_j,k_j}  ) dt_1dt_2\right)^2\!\!\!\!\leq B
$$
where $A$ and $B$ are as in the statement of the theorem.

Fix $t_2, \, \mu_{n_2,k_2}\in\R$; let $  g (t_1)=\int_\R f(t_1, t_2)e^{-2\pi i  b_2 n_2t_2  } \rect(t_2-a_2\mu_{n_2,k_2}  ) ) dt_2$. With this notation, the inequality  above  can be written as
\begin{equation}\label{e-ineq-fr-d2} A\leq  \sum_{(n_j,k_j)\in\Z^2}   \left(\int_{\R }g(t_1)e^{-2\pi i  b_1 n_1 t_1}\rect(t_1-a_1\mu_{n_1,k_1}  ) dt_1 \right)^2\leq B.
\end{equation}
By Theorem \ref{T-stability-rect}, the sets
$ {\mathcal B}_j= \left\{ e^{2\pi i b_j n_j t_j }\operatorname{rect}\left(t_j- a_j\mu_{n_j,k_j}  \right)\right\}_{k_j,n_j\in\mathbb Z}$, with $j=1,\, 2$,
are frames for $L^2(\R)$  with bounds $A_j\ge (1- 2( a_j b_j  {\mathcal L}_j)^{\frac 12})^2$, $B_j\leq (1+ 2( a_j b_j  {\mathcal L}_j)^{\frac 12})^2$.  Thus,
\begin{align}\label{e-ineq-fr-d3}
A_1 ||g ||_{L^2(\R)}^2& \leq\!\!\!\!\sum_{(n_1,k_1)\in\Z^2}\!\! \left(\int_{\R }g(t_1)e^{-2\pi i  b_1 n_1 t_1}\rect(a_1\mu_{n_1,k_1}-t_1 ) dt_1 \right)^2\\  \nonumber &\leq  B_1  ||g ||^2_{L^2(\R)} .
\end{align}
But
$
||g ||^2_{L^2(\R)}  = \int_\R \left(\int_\R f(t_1, t_2)e^{-2\pi i  b_2 n_2t_2  } \rect(t_2-a_2\mu_{n_2,k_2}  ) ) dt_2\right)^2 dt_1
$
and
\begin{align*}
\sum_{(n_2,k_2)\in\Z^2}\!\!\!\! ||g ||^2_{L^2(\R)}  &= \int_\R   \sum_{(n_2,k_2)\in\Z^2}  \!\!  \left(\int_\R f(t_1, t_2)e^{-2\pi i  b_2 n_2t_2  } \rect(t_2-a_2\mu_{n_2,k_2}  ) ) dt_2\right)^2\!\!\!  dt_1
\\ & \leq B_2 \int_\R \left(\int_\R |f(t_1, t_2)|^2 dt_2\right) dt_1= B_2.
\end{align*}
Similarly, $\sum_{(n_2,k_2)\in\Z^2} ||g ||^2_{L^2(\R)} \ge A_2$.
These inequalities   and \eqref{e-ineq-fr-d3} yield \eqref{e-ineq-fr-d2}.
\end{proof}

  A version of    Theorem \ref{T-d-dim} can be proved   for   functions $\operatorname{rect}^{(p)}  ({\bf x}) =\rect^{(p)}(x_1)$ $\rect^{(p)}(x_2)$  \dots \  $\rect^{(p)}(x_d)$, with $p\ge 2$.  We leave the details to the interested reader.

\section{Remarks  and open problems}

When $a=b=1$, the stability bound   in Theorem \ref{T-stability-rect} is  $L=\frac 14$. We do not know whether   $\frac 14$ can be replaced by any  larger constant or not.   A famous example by Ingham \cite{I} shows that the constant $\frac 14$ in Kadec's   theorem is optimal, but we could not   generalize  Ingham's example.

From  Theorem \ref{T-SunZ}  and   Plancherel's theorem, stability results for Riesz bases on spaces of  band-limited functions easily follow. Recall that
a function $f\in L^2(\R^d)$ is {\it band-limited } to a bounded measurable set $D$  (or:  {\it $f$ is in the Paley-Wiener space $PW_{D}$})  if its Fourier transform    vanishes outside $D$.
By   Plancherel's theorem,    $ PW_{D}$ is a closed subspace of $L^2(\R^d)$ which is isometrically isomorphic to $L^2(D)$.
The importance of exponential bases in the reconstruction of bandlimited functions is emphasized by the classical sampling theorem, attributed to Shannon, Whittaker, Kotel'nikov and others (\cite{Shan}, \cite{Uns}).
 By Theorem \ref{T-SunZ},  the set $\{e^{2\pi i \lambda_n t} \rect(t-k)\}_{ n\in\Z}$  is a Riesz basis of the subspace of $L^2(\R)$ spanned by  the functions $\{e^{2\pi i n t} \rect(t-k)\}_{ n \in\Z}$ whenever $|\lambda_n-n|<\ell<\frac 14$; thus,
     for every $k\in\Z$,  the set
  $\{e^{2\pi i k t} \sinc(t-\lambda_n)\}_{n\in\ZZ}$ is a Riesz basis of $PW_{(k-\frac 12,\ k+\frac 12)} $.
Our  Theorems \ref{T-stability-rect}   and \ref{T-stabi-n-hold} deal  with frames of functions with compact support, but do not yield  stability theorems in Paley-Wiener spaces. For example, consider  the set    $\B=\{e^{2\pi i n t}\rect(t-\mu_{n })\}_{n\in\Z}\subset \F$,
     with
   $\sup_{n\in\Z} | \mu_{n }|= L $;
    the set
   $\widehat \B= \{-e^{2\pi i \mu_{n }t}\sinc(t-n)\}_{n\in\Z}$  is  a subset of $PW_{(-\frac 12-L , \frac 12+L )}$, but  it is not a Riesz basis because $\B$  is not a Riesz  basis of $L^2(-\frac 12-L , \frac 12+L )$.

It would be interesting to study the stability properties of  Riesz bases or frames of $L^2(\R)$ in the form of $\{e^{2\pi i n t}\rect(T_kt-k)\}_{n,k\in\Z}$, where $\{T_k\}_{k\in\Z}$ is a   sequence of positive real numbers. We hope to address this problem in a subsequent paper.

\appendix
\section{The $p-$order holds}

Let $p\ge 0$ be an integer and   let $T>0$. A \emph{ $p-$order  hold}  (pOH) is a device which models  a sequence of impulses $\{q_n\}$ into a     piecewise polynomial function  of degree $p$.
   If   the   sequence    $\{q_n\}$  originates from a continuous-time signal, i.e. $q_n= f (nT)$  for some $f\in L^2(\R)$,  the function $f_p(t)$  obtained  through the $p-$order hold can be viewed   as an approximation of  the original function $f(t)$.
   The   zero-order hold (ZOH) is the simplest and most widely used model:
for a given  function $f\in C(\R)$, we let $f_{ZOH}(t)\,=\sum _{n=-\infty }^{\infty }f(nT)\ \operatorname{rect} \left( \frac  t T-n \right)$.
The zero order hold  model is unambiguously defined in the literature (\cite{Eshbach90}  \cite{Hin11} \cite{Opp14}) but the definition of   $p$-order hold varies. The extrapolation formulas that are most  used in pOH 	  are  discrete versions   of the Taylor expansion for differentiable function (\cite{Boni95})
but other interpolation polynomials can be considered.  When  $p\ge 1$ we can  let
$$f_{p}(t)\,= \sum_{n=-\infty}^{\infty} f(nT)\, \chi_{(Nt-\frac 12, NT+\frac 12)}\rect^{(p+1)} \left(\frac{t - nT}{T} \right).
$$
This model is    adopted, for example, in \cite[pg. 495]{Easton2010} for $p=1$.
The function  $f_{ZOH}(t)\,=\sum _{n=-\infty }^{\infty }f(\mu_n)\ \operatorname{rect} \left(\frac {t-  \mu_nT}{T}\right)$, where $\mu_n=n+\epsilon_n$, is considered in \cite{Ang09}. The term $\epsilon_n$ models the so-called \emph{timing jitter}, an unwelcome phenomenon of electronic systems.  {It is natural to investigate  whether signals in $L^2(\R)$ can effectively be reconstructed from ZOH devices with    jitter}. Similarly, we can  model the effect of the timing jitter also in pOH devices as $f_{pOH}(t)\,=\sum _{n=-\infty }^{\infty }f(\mu_n)\ \operatorname{rect}^{(p)} \left(\frac {t-  \mu_nT}{T}\right)$.
Our Theorems \ref{T-stability-rect} and \ref{T-stabi-n-hold} can be used to estimate how much jitter can be   {allowed} in ZOH and pOH models. {  For example, suppose that  in a ZOH model
a given signal $ f\in L^2(\R)$ is approximated with linear combinations of  $ \rect(t-k)$, with ${ k\in\Z}$.
By Theorem \ref{T-stability-rect} (with $a=b=1$), the set $\B=\{e^{2\pi i n t} \rect(k-t)\}_{k\in\Z, n\ne 0}\cup\{ \rect(\mu_{0,k}-t)\}_{k\in\Z}$ is a Riesz basis in $L^2(\R)$ if $|\epsilon_k|=|\mu_{0,k}-k|<L<\frac 14$.  { If the timing jitter $\epsilon_k$ satisfies this inequality},    signals in $L^2(\R)$ can be effectively reconstructed  from the functions in $\B$.

In \cite{Lim90} an extension of the ZOH model in dimension $d= 2$ is considered. Let $f(t_1, t_2)$ be an analog signal, and let $\rect(t_1,t_2)= \rect(t_1)\, \rect(t_2)$. We consider a sequence of equally spaced points $(T_1 n_1,\ T_2 n_2)$, where $T_1, \ T_2>0$ and $n_1, n_2\in\Z$. The zero-hold reconstruction of the signal $f$ is
$$f_{ZOH}(t_1,t_2)\,= \sum _{(n_1, \, n_2)\in\Z^2 } f(n_1, n_2)\ \operatorname{rect} \left(\frac {t_1-n_1 T_1}{T_1},\ \frac {t_2 -n_2T_2}{T_2}\right).  $$
In the presence of  jitter, the sampling points ${\bf n}=(n_1, n_2)$ are replaced by
${\bf \mu_{ n} }= (n_1+\epsilon_{n_1},\   n_2+\epsilon_{n_2})$; we can use  Theorem \ref{T-d-dim} to conclude that {  if the timing jitter $(\epsilon_{n_1},\, \epsilon_{n_2}) $ satisfies    the inequality:  $\max\{ |\epsilon_{n_1}|, \ |\epsilon_{n_1}|\}<\ell<\frac 14$,  signals in $L^2(\R^2)$ can be effectively reconstructed from functions in $\B=\{e^{2\pi i \langle {\bf n}, \ {\bf t}\rangle } \rect({\bf k-t})\}_{{\bf k}\in\Z^2, {\bf n}\ne 0}\cup\{ \rect({\bf \mu_{0,k}-t})\}_{{\bf k}\in\Z^2}$.
}

Jitter  can appear also in the orthogonal frequency-division multiplexing (OFDM), a method of encoding digital data on multiple carrier frequencies. OFDM has developed into a popular scheme for wideband digital communication, used in applications such as digital television and audio broadcasting. According to the basic OFDM realization \cite{Hrasnica},\cite{Rohl}, the transmitted signal $f(t)$ can be often expressed by
$$f(t)=\sum_{n=0}^{N-1}\sum_{k\in\mathbb Z}\alpha_{n,k}\, \rect(t-k)\, e^{ 2\pi i \lambda_n t}$$
  %$N$ is the number of subcarrier frequencies while $T$ is the OFDM symbol duration.
  where $\lambda_n=n/T$. The \emph{frequency jitter } is modeled by $\lambda_n = \frac{n+\epsilon_n}{T}$, for some  $\epsilon_n>0$. As in the previous example, it is important to understand how much jitter can be tolerated  in order to   obtain a good signal reconstruction.

Let $\Gamma=\{0,1,\dots,N-1\}$.
By Theorem \ref{T-SunZ} (with $a=b=1$), the set $\mathcal C=\{e^{2\pi i n t} \rect(k-t)\}_{n,k\in\Z, n\notin \Gamma}\cup\{e^{2\pi i \lambda_n t} \rect(k-t)\}_{k\in\Z, n\in \Gamma}$ is a Riesz basis in $L^2(\R)$ if $|\epsilon_n|=|\lambda_n-n|<L<\frac 14$, and so signals in $L^2(\R)$ can be effectively reconstructed  also  in terms of the functions in $\mathcal C$.

\medskip

%\code{\section*{Acknowledgments}}

%\section*{References}

\bibliographystyle{te}
\bibliography{references}

\begin{thebibliography}{34}
\newcommand{\enquote}[1]{``#1''}
\providecommand{\natexlab}[1]{#1}
\providecommand{\url}[1]{\texttt{#1}}
\providecommand{\urlprefix}{URL }
\providecommand{\bibAnnoteFile}[1]{%
  \IfFileExists{#1}{\begin{quotation}\noindent\textsc{Key:} #1\\
  \textsc{Annotation:}\ \input{#1}\end{quotation}}{}}
\providecommand{\bibAnnote}[2]{%
  \begin{quotation}\noindent\textsc{Key:} #1\\
  \textsc{Annotation:}\ #2\end{quotation}}

\bibitem[{Angrisani and D'Arco(2009)}]{Ang09}
Angrisani, L. and M.~D'Arco (2009), \enquote{Modeling timing jitter effects in
  digital-to-analog converters.} \emph{IEEE Transactions on Instrumentation and
  Measurement}, 58, 330--336.
\bibAnnoteFile{Ang09}

\bibitem[{Antony~Selvan and Radha(2016)}]{SR}
Antony~Selvan, A. and R.~Radha (2016), \enquote{Sampling and reconstruction in
  shift invariant spaces of b-spline functions.} \emph{Acta Applicandae
  Mathematicae}, 145, 175--192.
\bibAnnoteFile{SR}

\bibitem[{Balan et~al.(2006)Balan, Casazza, Heil, and Landau}]{BCHL}
Balan, Radu, Peter~G. Casazza, Christopher Heil, and Zeph Landau (2006),
  \enquote{Density, overcompleteness, and localization of frames. ii. {G}abor
  systems.} \emph{Journal of Fourier Analysis and Applications}, 12, 307--344.
\bibAnnoteFile{BCHL}

\bibitem[{Bonivento et~al.(1995)Bonivento, Melchiorri, and Zanasi}]{Boni95}
Bonivento, Claudio, Claudio Melchiorri, and Roberto Zanasi (1995),
  \emph{Sistemi di controllo digitale}. Esculapio.
\bibAnnoteFile{Boni95}

\bibitem[{Brezis(2011)}]{B}
Brezis, Haim (2011), \emph{Functional analysis, {S}obolev spaces and partial
  differential equations}. Springer Verlag.
\bibAnnoteFile{B}

\bibitem[{Christensen(1995)}]{C2}
Christensen, Ole (1995), \enquote{A {P}aley-{W}iener theorem for frames.}
  \emph{Proceedings of the American Mathematical Society}, 123, 2199--2201.
\bibAnnoteFile{C2}

\bibitem[{Christensen(1996)}]{C}
Christensen, Ole (1996), \enquote{Moment problems and stability results for
  frames with applications to irregular sampling and {G}abor frames.}
  \emph{Applied and Computational Harmonic Analysis}, 3, 82 -- 86.
\bibAnnoteFile{C}

\bibitem[{Christensen(2003)}]{C1}
Christensen, Ole (2003), \emph{An introduction to frames and {R}iesz bases},
  volume~7. Birkh\"auser.
\bibAnnoteFile{C1}

\bibitem[{Daubechies(1992)}]{D}
Daubechies, I. (1992), \emph{Ten Lectures on Wavelets}. Society for Industrial
  and Applied Mathematics.
\bibAnnoteFile{D}

\bibitem[{Easton~Jr(2010)}]{Easton2010}
Easton~Jr, Roger~L (2010), \emph{Fourier methods in imaging}. John Wiley \&
  Sons.
\bibAnnoteFile{Easton2010}

\bibitem[{Eshbach et~al.(1990)Eshbach, Tapley, and Poston}]{Eshbach90}
Eshbach, Ovid~Wallace, Byron~D Tapley, and Thurman~R. Poston (1990),
  \emph{Eshbach's handbook of engineering fundamentals}. John Wiley \& Sons.
\bibAnnoteFile{Eshbach90}

\bibitem[{Favier and Zalik(1995)}]{Favier95}
Favier, S.J. and R.A. Zalik (1995), \enquote{On the stability of frames and
  {R}iesz bases.} \emph{Applied and Computational Harmonic Analysis}, 2, 160 --
  173.
\bibAnnoteFile{Favier95}

\bibitem[{Feichtinger and Kaiblinger(2004)}]{FG}
Feichtinger, Hans and Norbert Kaiblinger (2004), \enquote{Varying the
  time-frequency lattice of {G}abor frames.} \emph{Transactions of the American
  Mathematical Society}, 356, 2001--2023.
\bibAnnoteFile{FG}

\bibitem[{Feichtinger and Sun(2006)}]{FS}
Feichtinger, Hans~G. and Wenchang Sun (2006), \enquote{Stability of {G}abor
  frames with arbitrary sampling points.} \emph{Acta Mathematica Hungarica},
  113, 187--212.
\bibAnnoteFile{FS}

\bibitem[{Gr{\"o}chenig(1993)}]{G}
Gr{\"o}chenig, Karlheinz (1993), \enquote{Irregular sampling of wavelet and
  short-time {F}ourier transforms.} \emph{Constructive Approximation}, 9,
  283--297.
\bibAnnoteFile{G}

\bibitem[{Gr{\"o}chenig(2001)}]{G2}
Gr{\"o}chenig, Karlheinz (2001), \emph{Foundations of time-frequency analysis}.
  Birkh\"auser.
\bibAnnoteFile{G2}

\bibitem[{Heil(2011)}]{He10}
Heil, Christopher (2011), \emph{A basis theory primer}. Birkh\"auser.
\bibAnnoteFile{He10}

\bibitem[{Hinrichsen and Pritchard(2005)}]{Hin11}
Hinrichsen, Diederich and Anthony~J Pritchard (2005), \emph{Mathematical
  systems theory I: modelling, state space analysis, stability and robustness},
  volume~48. Springer Berlin.
\bibAnnoteFile{Hin11}

\bibitem[{Hrasnica et~al.(2005)Hrasnica, Haidine, and Lehnert}]{Hrasnica}
Hrasnica, Halid, Abdelfatteh Haidine, and Ralf Lehnert (2005), \emph{Broadband
  powerline communications: network design}. John Wiley \& Sons.
\bibAnnoteFile{Hrasnica}

\bibitem[{Ingham(1936)}]{I}
Ingham, A.~E. (1936), \enquote{Some trigonometrical inequalities with
  applications to the theory of series.} \emph{Mathematische Zeitschrift}, 41,
  367--379.
\bibAnnoteFile{I}

\bibitem[{Lim(1990)}]{Lim90}
Lim, Jae~S (1990), \emph{Two-dimensional signal and image processing}. Prentice
  Hall.
\bibAnnoteFile{Lim90}

\bibitem[{Mishchenko(2010)}]{M}
Mishchenko, E.~V. (2010), \enquote{Determination of {R}iesz bounds for the
  spline basis with the help of trigonometric polynomials.} \emph{Siberian
  Mathematical Journal}, 51, 660--666.
\bibAnnoteFile{M}

\bibitem[{Oppenheim et~al.(2014)Oppenheim, Willsky, and Nawab}]{Opp14}
Oppenheim, A.V., A.S. Willsky, and S.H. Nawab (2014), \emph{Signals and
  systems}. Pearson.
\bibAnnoteFile{Opp14}

\bibitem[{Prautzsch et~al.(2002)Prautzsch, Boehm, and Paluszny}]{PBP}
Prautzsch, Hartmut, Wolfgang Boehm, and Marco Paluszny (2002), \emph{B{\'e}zier
  and {B}-spline techniques}. Springer Science \& Business Media.
\bibAnnoteFile{PBP}

\bibitem[{Rohling(2011)}]{Rohl}
Rohling, Hermann (2011), \emph{OFDM: concepts for future communication
  systems}. Springer Science \& Business Media.
\bibAnnoteFile{Rohl}

\bibitem[{Schoenberg(1988)}]{Sch1}
Schoenberg, I.~J. (1988), \emph{Contributions to the Problem of Approximation
  of Equidistant Data by Analytic Functions}, 3--87. Birkh{\"a}user Boston.
\bibAnnoteFile{Sch1}

\bibitem[{Schoenberg(1969)}]{Sch2}
Schoenberg, I.J (1969), \enquote{Cardinal interpolation and spline functions.}
  \emph{Journal of Approximation Theory}, 2, 167 -- 206.
\bibAnnoteFile{Sch2}

\bibitem[{Shannon(1949)}]{Shan}
Shannon, C.~E. (1949), \enquote{Communication in the presence of noise.}
  \emph{Proceedings of the IRE}, 37, 10--21.
\bibAnnoteFile{Shan}

\bibitem[{Sun and Zhou(1999)}]{Sun99}
Sun, Wenchang and Xingwei Zhou (1999), \enquote{On {K}adec's 1/4-theorem and
  the stability of {G}abor frames.} \emph{Applied and Computational Harmonic
  Analysis}, 7, 239 -- 242.
\bibAnnoteFile{Sun99}

\bibitem[{Sun and Zhou(2001)}]{Sun01}
Sun, Wenchang and Xingwei Zhou (2001), \enquote{On the stability of {G}abor
  frames.} \emph{Advances in Applied Mathematics}, 26, 181 -- 191.
\bibAnnoteFile{Sun01}

\bibitem[{Sun and Zhou(2003)}]{SZ2}
Sun, Wenchang and Xingwei Zhou (2003), \enquote{Irregular {G}abor frames and
  their stability.} \emph{Proceedings of the American Mathematical Society},
  131, 2883--2893.
\bibAnnoteFile{SZ2}

\bibitem[{Unser(2000)}]{Uns}
Unser, M. (2000), \enquote{Sampling-50 years after {S}hannon.}
  \emph{Proceedings of the IEEE}, 88, 569--587.
\bibAnnoteFile{Uns}

\bibitem[{Unser et~al.(1993)Unser, Aldroubi, and Eden}]{UAE}
Unser, M., A.~Aldroubi, and M.~Eden (1993), \enquote{B-spline signal
  processing. i. theory.} \emph{IEEE Transactions on Signal Processing}, 41,
  821--833.
\bibAnnoteFile{UAE}

\bibitem[{Young(2001)}]{You}
Young, Robert~M (2001), \emph{An Introduction to Non-Harmonic {F}ourier
  Series}. Academic Press.
\bibAnnoteFile{You}

\end{thebibliography}

\end{document}